\documentclass[12pt,a4paper]{amsart}
\usepackage{amsmath}
\usepackage{amscd}
\usepackage{drpack}
\usepackage[english]{babel}
\usepackage{verbatim}
\usepackage[shortlabels]{enumitem}
\usepackage{color}
\usepackage{tikz-cd}
\usepackage{mathrsfs}

\newtheoremstyle{mio}%
{}{} 
{\itshape}{} 
{\bfseries}{.}{ } 
{#1 #2\thmnote{~\mdseries(#3)}} 
\theoremstyle{mio}
\newtheorem{teor}{Theorem}[section]
\newtheorem{cor}[teor]{Corollary}
\newtheorem{prop}[teor]{Proposition}
\newtheorem{lemma}[teor]{Lemma}
\newtheorem{defin}[teor]{Definition}

\newtheoremstyle{definition2}%
{}{} 
{}{} 
{\bfseries}{.}{ } 
{#1 #2\thmnote{\mdseries~ #3}} 
\theoremstyle{definition2}

\newtheorem{oss}[teor]{Remark}

\newcommand{\Inv}{\mathrm{Inv}}
\newcommand{\Div}{\mathrm{Div}}
\newcommand{\Int}{\mathrm{Int}}

\DeclareMathOperator{\supp}{supp}

\newcommand{\inverse}{\mathrm{inv}}
\newcommand{\V}{\mathcal{V}}

\newcommand{\insfunct}{\mathcal{F}}
\newcommand{\insbound}{\insfunct_b}
\newcommand{\inscont}{\mathcal{C}}

\newcommand{\inscrit}{\mathrm{Crit}}
\newcommand{\critx}[1]{\phantom{}^{#1}\inscrit}
\newcommand{\bcrit}{\critx{\omega}}

\newcommand{\cl}{\mathrm{cl}}
\newcommand{\qz}{\mathcal{H}}
\newcommand{\mmax}{\mathcal{M}}

\title{Residual functions and divisorial ideals}
\author{Dario Spirito}
\date{\today}
\address{Dipartimento di Scienze Matematiche, Informatiche e Fisiche, Universit\`a degli Studi di Udine, Udine, Italy}
\email{dario.spirito@uniud.it}

\begin{document}
\begin{abstract}
We define a \emph{residual function} on a topological space $X$ as a function $f:X\longrightarrow\insZ$ such that $f^{-1}(0)$ contains an open dense set, and we use this notion to study the freeness of the group of divisorial ideals on a Pr\"ufer domain.
\end{abstract}

\maketitle

\section{Introduction}
Let $X$ be a topological space. We say that a function $f:X\longrightarrow\insZ$ is \emph{residual} if $f^{-1}(0)$ contains an open dense set; thus, residual function are in a sense function that are zero almost everywhere in a topological sense. We show that residual functions form a group $\qz(X)$ and that the subgroup $\qz_b(X)$ of bounded residual functions is a direct summand of the (free) group of all bounded functions $X\longrightarrow\insZ$.

We then use this notion to analyze the group $\Div(D)$ of $v$-invertible divisorial ideals of a Pr\"ufer domain $D$. Divisorial ideals are a class of fractional ideals of an integral domain whose study is rooted in the work of Krull \cite{krull_breitage_I-II}, and that is the most important case of \emph{star operations}, a class of closure operations strongly linked to the multiplicative structure of the ideals of $D$ (see \cite[Chapter 32]{gilmer}).

In the case of Pr\"ufer completely integrally closed domains, the set of divisorial ideals form a group $\Div(D)$ that can be interpreted as the completion (in the sense of lattice-ordered groups) of the group $\Inv(D)$ of invertible ideals \cite[Proposition 3.1]{HK-Olb-Re}. For the case of SP-domains (domains where every ideals can be written as a product of radical ideals) this fact allows us to show that, as $\Inv(D)$ is isomorphic to the group of continuous functions from the maximal space $\mmax$ (endowed with the inverse topology) to $\insZ$, $\Div(D)$ is isomorphic to the group of continuous function from the Gleason cover $E_\mmax$ of $\mmax$ (see \cite[Theorems 5.1 and 5.3]{HK-Olb-Re} and \cite[Corollary 4.2]{SP-scattered}); in particular, the groups $\Inv(D)$, $\Div(D)$ and $\Inv(D)/\Div(D)$ are all free \cite[Theorem 6.5]{SP-scattered}. In \cite{InvXD}, this case was generalized to the case of the subgroups $\Inv_X(D)$ and $\Div_X(D)$ of ideals whose support is contained in some $X\subseteq\Max(D)$, under the assumption that $D$ is a strongly discrete Pr\"ufer domain and that $X$ contains no critical maximal ideals.

Divisorial ideals also enjoy a topological characterization. In \cite{fin-lop-divisorial-intD}, the authors showed that, in the case of the ring of integer-valued polynomials over a discrete valuation domain $V$, the divisoriality of radical unitary ideals can be checked at the topological level by considering the subspace of the completion $\widehat{V}$ associated to the ideal. Following this approach, in this paper we give an interpretation of the results of \cite{InvXD} that is more directly topological, and does not need to use the language of lattice-ordered groups: we show that, under some assumptions on $X$, the difference between the ideal function of an ideal $I$ and its divisorial closure $I^v$ is residual, and thus that if every ideal function is bounded (so that there is an injection from $\Inv(D)$ to the group $\insfunct_b(X,\insZ)$ of bounded functions $X\longrightarrow\insZ$) then we have an injection from $\Div(D)$ to $\insfunct_b(X,\insZ)/\qz_b(X)$.

\section{Preliminaries}
\subsection{Topology}
Let $X$ be a topological space. If $Y\subseteq X$, we denote by $\cl(Y)$ the closure of $Y$.

A subset $Y\subseteq X$ is \emph{dense} if $\cl(Y)=X$, and is \emph{nowhere dense} if $\cl(Y)$ has empty interior. In particular, $Y$ is nowhere dense if and only if $X\setminus Y$ contains an open dense set.

We denote by $\inscont(X,\insZ)$ the group of continuous functions from $X$ to $\insZ$ (where $\insZ$ is endowed with the discrete topology and the operation is the sum), by $\inscont_b(X,\insZ)$ the subgroup of bounded functions and by $\inscont_c(X,\insZ)$ the subgroup of continuous functions with compact support, where the \emph{support} of $f$ is the closure of $f^{-1}(\insZ\setminus\{0\})$.

\subsection{Free groups}
All groups considered in this paper are abelian; by a \emph{free group} we always mean a free abelian group, i.e., an abelian group with a basis. Every subgroup of a free group is free.

Let $X$ be a set. For a subset $Y\subseteq X$, we denote by $\chi_Y$ the characteristic function of $Y$, i.e., $\chi_Y(a)=1$ if $a\in Y$ while $\chi_Y(a)=0$ if $a\notin Y$.

Let $\insbound(X,\insZ)$ be the group of bounded function from $X$ to $\insZ$, where the operation is the componentwise addition. A subgroup $G\subseteq\insbound(X,\insZ)$ is a \emph{Specker group} if, for every $f\in G$, we have $\chi_{f^{-1}(n)}\in G$ for every $n\inN$ \cite[\textsection 1]{nobeling}. If $H\subseteq G$ are Specker groups, then $H$ is a direct summand of $G$, and the quotient $G/H$ is free \cite[Satz 2]{nobeling}. In particular, the group $\insbound(X,\insZ)$ is free.

\subsection{Divisorial ideals}
Let $D$ be an integral domain with quotient field $K$. A \emph{fractional ideal} of $D$ is a subset $I\subseteq K$ such that $dI\subseteq D$ for some $d\in K$, $d\neq 0$. Following \cite{InvXD}, the \emph{support} of a fractional ideal $I$ is the set
\begin{equation*}
\supp(I):=\{P\in\Spec(D)\mid ID_P\neq D_P\};
\end{equation*}
if $I\subseteq D$ is a proper ideal, then $\supp(I)=\V(I)$ is the closed set associated to $I$.

A fractional ideal $I$ is \emph{invertible} if there is a fractional ideal $J$ such that $IJ=D$; in this case, $J$ is equal to $(D:I):=\{d\in K\mid dI\subseteq D\}$. The set of invertible ideals, under product, is a group denoted by $\Inv(D)$.

The \emph{divisorial closure} of a fractional ideal $I$ is $I^v:=I=(D:(D:I))$, and $I$ is \emph{divisorial} if $I=I^v$. An ideal is \emph{$v$-invertible} if $(I(D:I))^v=D$. The set of divisorial $v$-invertible ideals is a group under the \emph{$v$-product} $I\times_vJ:=(IJ)^v$, denoted by $\Div(D)$. Every invertible ideal is divisorial and $v$-invertible, and the inclusion map $\Inv(D)\longrightarrow\Div(D)$ is an injective group homomorphism.

\subsection{Pr\"ufer domains}
A \emph{valuation domain} is an integral domain $V$ such that, for every $x$ in the quotient field $K$ of $V$, at least one of $x$ and $x^{-1}$ is in $V$. To every valuation domain we can associate a totally ordered abelian group $\Gamma_v$ (the \emph{value group} of $V$) and a surjective map $v:K\setminus\{0\}\longrightarrow\Gamma_v$ (the \emph{valuation} relative to $V$) that satisfy $v(x+y)\geq\min\{v(x),v(y)\}$ and $v(xy)=v(x)+v(y)$ for all $x,y\in K$ (and $x+y\neq 0$ for the first property). A valuation domain $V$ is Noetherian if and only if its value group is isomorphic to $\insZ$; in this case, $V$ is said to be a \emph{discrete valuation ring} (DVR).

A \emph{Pr\"ufer domain} is an integral domain $D$ that is locally a valuation domain, i.e., such that $D_P$ is a valuation domain for every $P\in\Spec(D)$. Pr\"ufer domains can be characterized in many ways; for example, $D$ is a Pr\"ufer domain if and only if every finitely generated ideal is invertible. See \cite[Chapter IV]{gilmer} and \cite[Theorem 1.1.1]{fontana_libro} for more characterizations.

A Pr\"ufer domain is \emph{strongly discrete} if no nonzero prime ideal is idempotent (i.e., if $P\neq P^2$ for every prime ideal $P\neq(0)$). Equivalently, a Pr\"ufer domain $D$ is strongly discrete if and only if $PD_P$ is a principal ideal of $D_P$ for every $P\in\Spec(D)$ \cite[Section 5.3]{fontana_libro}.

\subsection{Topologies}
Let $R$ be a ring. The \emph{inverse topology} on $\Spec(R)$ is the topology generated by the closed sets $\V(I)$, as $I$ ranges among the finitely generated ideals of $R$. The \emph{constructible topology} on $\Spec(R)$ is the coarsest topology that is finer than both the Zariski and the inverse topology. Both topologies and compact, and the constructible topology is also Hausdorff. On the maximal space $\Max(R)$, the inverse and the constructible topology coincide \cite[Corollary 4.4.9(i)]{spectralspaces-libro}.

If $X\subseteq\Spec(D)$, we denote by $X^\inverse$ the set $X$ endowed with the inverse topology.

\section{Residual functions}\begin{defin}
Let $X$ be a topological space. We say that a function $f:X\longrightarrow\insZ$ is \emph{residual} if $f^{-1}(0)$ contains an open dense subset. We denote by $\qz(X)$ the set of residual functions on $X$, and by $\qz_b(X)$ the set of bounded residual functions.
\end{defin}

\begin{oss}
~\begin{enumerate}
\item Residual functions are ``mostly zero'', in the sense that they are different from $0$ only on a ``small'' (i.e., nowhere dense) set.
\item No nonzero continuous function is residual. Indeed, if $f:X\longrightarrow\insZ$ is continuous and nonzero, there is  an $n\neq 0$ such that $f^{-1}(n)$ is nonempty; since $f$ is continuous and $\insZ$ is discrete, $f^{-1}(n)$ is clopen, and in particular $f^{-1}(0)$ cannot be dense. Hence $f$ is not residual.
\end{enumerate}
\end{oss}

\begin{lemma}\label{lemma:residual}
Let $X$ be a topological space and $f:X\longrightarrow\insZ$. Then, the following hold.
\begin{enumerate}[(a)]
\item\label{lemma:residual:now} $f$ is residual if and only if $f^{-1}(\insZ\setminus\{0\})$ is nowhere dense.
\item\label{lemma:residual:bound} If $f$ is bounded, then $f$ is residual if and only if $f^{-1}(n)$ is nowhere dense for every $n\neq 0$.
\end{enumerate}
\end{lemma}
\begin{proof}
\ref{lemma:residual:now} If $f$ is residual, then $f^{-1}(0)$ contains an open dense set $\Omega$. In particular, the closure of $Y:=f^{-1}(\insZ\setminus\{0\})$ is contained in $X\setminus\Omega$, which by the density of $\Omega$ cannot contain an open set. Hence $Y$ is nowhere dense.

Conversely, if $Y$ is nowhere dense, then $\Omega:=X\setminus\cl(Y)$ is an open set contained in $f^{-1}(0)$, which is dense since $\cl(Y)$ is nowhere dense. The claim is proved.

\ref{lemma:residual:bound} If $f$ is residual, each $f^{-1}(n)$ is nowhere dense since it is contained in $f^{-1}(\insZ\setminus\{0\})$, which is nowhere dense by the previous point.

Conversely, suppose that each $f^{-1}(n)$ is nowhere dense. Since $f$ is bounded, only finitely many such sets are nonempty; hence, $f^{-1}(\insZ\setminus\{0\}):=\bigcup_{n\neq 0}f^{-1}(n)$ is a finite union of nowhere dense subsets, and thus it is nowhere dense \cite[Problem 25A]{willard_gentop}. By the previous point, $f$ is residual.
\end{proof}

\begin{prop}
The sets $\qz(X)$ and $\qz_b(X)$ are subgroups of $\insfunct(X,\insZ)$.
\end{prop}
\begin{proof}
If $f\in\qz(X)$, then $(-f)^{-1}(0)=f^{-1}(0)$ and thus $-f$ is residual. Moreover, if also $g\in\qz(X)$, then $(f+g)^{-1}(0)\supseteq f^{-1}(0)\cap g^{-1}(0)$; since the intersection of two dense open sets is dense, $f^{-1}(0)\cap g^{-1}(0)$ contains an open dense set and thus $f+g$ is residual. Hence $\qz(X)$ and $\qz_b(X)=\qz(X)\cap\insbound(X,\insZ)$ are subgroups of $\insfunct(X,\insZ)$.
\end{proof}

\begin{teor}\label{teor:quoz-free}
Let $X$ be a topological space. Then, the quotient groups
\begin{equation*}
\frac{\insbound(X,\insZ)}{\qz_b(X)}\quad\text{~and~}\quad\frac{\insbound(X,\insZ)}{\qz_b(X)\oplus\inscont_b(X,\insZ)}
\end{equation*}
are free.
\end{teor}
\begin{proof}
Let $\mathbf{1}$ be the identity function on $X$, and let $G:=\qz_b(X)+\mathbf{1}\insZ$. We claim that $G$ is a Specker group. Let $f\in G$: then, there are $h\in\qz_b(X)$ and $n\inN$ such that $f=h+n\mathbf{1}$, and for any $k\inN$ we have $f^{-1}(k)=h^{-1}(k-n)$.

If $k\neq n$, then $h^{-1}(k-n)$ is nowhere dense, and thus $\chi_{f^{-1}(k)}$ is residual, and in particular it belongs to $G$. If $k=n$, then $h^{-1}(0)$ contains an open dense subset, and thus $X\setminus h^{-1}(0)$ is nowhere dense. Thus, its characteristic function is residual, and so
\begin{equation*}
\chi_{f^{-1}(n)}=\mathbf{1}-\chi_{X\setminus h^{-1}(0)}\in G.
\end{equation*}
Therefore, every $\chi_{f^{-1}(k)}$ belongs to $G$, and thus $G$ is Specker. It follows that $\insbound(X,\insZ)=G\oplus H$ for some free group $H$. Moreover, $\qz_b(X)\cap\mathbf{1}\insZ=(0)$, and thus $G$ is the direct sum of $\qz_b(X)$ and $\mathbf{1}\insZ$: hence, $\insbound(X,\insZ)=\qz_b(X)\oplus \mathbf{1}\insZ\oplus H$, and so $\insbound(X,\insZ)/\qz_b(X)$ is isomorphic to the free group $\mathbf{1}\insZ\oplus H$.

\bigskip

We now show that $G':=\qz_b(X)\oplus\inscont_b(X,\insZ)$ is Specker. Let $f\in G'$: then, $f=h+g$ with $h\in\qz_b(X)$ and $g\in\inscont_b(X,\insZ)$. Fix $n\inZ$, and let $Y_a:=h^{-1}(a)\cap g^{-1}(n-a)$. Then, 
\begin{equation*}
f^{-1}(n)=\bigcup_{a+b=n}Y_a.
\end{equation*}
Moreover, the union above is disjoint (as $h^{-1}(a)$ and $h^{-1}(a')$ are disjoint if $a\neq a'$) and only finitely many of the intersections are nonempty (as $h,g$ are bounded). Hence,
\begin{equation*}
\chi_{f^{-1}(n)}=\sum_{a\inZ}\chi_{Y_a},
\end{equation*}
and it is enough to show that each $\chi_{Y_a}$ belongs to $G'$. If $a\neq 0$, the set $Y_a$ is nowhere dense since it is contained in the nowhere dense set $h^{-1}(a)$, and thus its characteristic function is residual (i.e., in $\qz_b(X)$). If $a=0$, then
\begin{equation*}
Y_0=g^{-1}(n)\setminus(g^{-1}(n)\setminus h^{-1}(0))=g^{-1}(n)\setminus(g^{-1}(n)\cap h^{-1}(\insZ\setminus\{0\})).
\end{equation*}
Let $Y':=g^{-1}(n)\cap h^{-1}(\insZ\setminus\{0\})$. Then, $Y'\subseteq g^{-1}(n)$, and thus
\begin{equation*}
\chi_{Y_0}=\chi_{g^{-1}(n)}-\chi_{Y'}.
\end{equation*}
Since $g$ is continuous, $g^{-1}(n)$ is clopen and $\chi_{g^{-1}(n)}$ is continuous; on the other hand, $Y'\subseteq h^{-1}(\insZ\setminus\{0\})$ is nowhere dense, and thus its characteristic function is residual. Hence also $\chi_{Y_0}\in G'$, and so $\chi_{f^{-1}(n)}\in G'$. Therefore, $G'$ is Specker and $\insbound(X,\insZ)/G'$ is free.
\end{proof}

\section{Algebraic applications}
Throughout the section, let $D$ be a Pr\"ufer domain. As done in \cite{InvXD} for invertible ideals, we want to associate to certain subsets $X\subseteq\Max(D)$ a subgroup $\Div_X(D)$ of $\Div(D)$, and to each element of $\Div_X(D)$ a function $X\longrightarrow\insZ$. 
\begin{defin}
Let $D$ be an integral domain and let $X\subseteq\Spec(D)$. We define $\Div_X(D)$ as the subgroup of $\Div(X)$ of the ideals $I\in\Div(D)$ such that $\supp(I),\supp((D:I))\subseteq X$.
\end{defin}

Consider now a non-idempotent prime ideal $P$ of $D$. Then, $Q:=\bigcap_{n\geq 1}P$ is a prime ideal of $D$, and it is the largest prime properly contained in $P$; therefore, the quotient $D_P/QD_P$ is a one-dimensional valuation domain whose maximal ideal $PD_P/QD_P$ is not idempotent. Hence, $D_P/QD_P$ is a discrete valuation ring, and so it carries a discrete valuation $v'_P$. In particular, if $I$ is a fractional ideal of $D$ such that $QD_P\subsetneq ID_P\subseteq D_Q$, we can define a value
\begin{equation*}
w_P(I):=v'_P(ID_P/QD_P)\in\insZ.
\end{equation*}
If this happens for every $P\in X$, we say that $I$ is \emph{evaluable} on $X$. In particular, if $X\subseteq\Max(D)$ is open (with respect to the inverse topology) and $I\in\Div_X(D)$ then $I$ is evaluable on $X$ \cite[Lemma 5.3]{InvXD}. Therefore, for every $I\in\Inv_X(D)$ we have a well-defined \emph{ideal function}
\begin{equation*}
\begin{aligned}
\nu_I\colon X & \longrightarrow \insZ,\\
P & \longmapsto w_P(I).
\end{aligned}
\end{equation*}

\begin{lemma}\label{lemma:nuI}
Let $X\subseteq\Max(D)$, and let $I,J$ be fractional ideals that are evaluable on $X$. Then, the following hold.
\begin{enumerate}[(a)]
\item $IJ$ is evaluable on $X$, and $\nu_{IJ}=\nu_I+\nu_J$.
\item If $I,J$ are proper ideals with $\V(I),\V(J)\subseteq X$ and $\nu_I=\nu_J$, then $I=J$.
\end{enumerate}
\end{lemma}
\begin{proof}
See Lemmas 5.2 and 5.4(a) of \cite{InvXD}.
\end{proof}

Therefore, we can define a map
\begin{equation*}
\begin{aligned}
\Psi_0\colon \Div_X(D) & \longrightarrow \insfunct(X,\insZ),\\
I & \longmapsto \nu_I.
\end{aligned}
\end{equation*}
This map, however, is \emph{not} a group homomorphism. Indeed, if $I,J\in\Div_X(D)$, by Lemma \ref{lemma:nuI} we have $\nu_{IJ}=\nu_I+\nu_J$; however, the product $I\times_vJ$ is not equal to $IJ$ but rather to $(IJ)^v$, and $\nu_{IJ}\neq\nu_{(IJ)^v}$ (at least if $I,J\subseteq D$, by Lemma \ref{lemma:nuI}).

To understand how far are $\nu_I$ and $\nu_{I^v}$, we define the following set.
\begin{defin}
Let $D$ be a strongly discrete Pr\"ufer domain and let $X\subseteq\Max(D)$. Let $I$ be a fractional ideal that is evaluable on $X$. We set
\begin{equation*}
\lambda_X(I):=\{M\in X\mid \nu_{I^v}(M)\neq\nu_I(M)\}.
\end{equation*}
\end{defin}

\begin{defin}
\cite[Definition 5.7]{InvXD} Let $D$ be a Pr\"ufer domain, and let $X\subseteq\Max(D)$. We say that $P\in X$ is:
\begin{itemize}
\item \emph{critical} (with respect to $X$) if there is no radical ideal $I\in\Inv_X(D)$ such that $I\subseteq P$;
\item \emph{bounded-critical} (with respect to $X$) if there is no ideal $I\in\Inv_X(D)$ such that $I\subseteq P$ and $\nu_I$ is bounded.
\end{itemize}
We denote by $\inscrit_X(D)$ and $\bcrit_X(D)$, respectively, the set of critical ideals and bounded-critical ideals with respect to $X$.
\end{defin}

\begin{lemma}
Let $D$ be a Pr\"ufer domain and let $X\subseteq\Max(D)$. Let $I,J$ be fractional ideals of $D$ that are evaluable on $X$. Then, $\lambda_X(I)\subseteq\lambda_X(IJ)$. If $J$ is invertible, then $\lambda_X(I)=\lambda_X(IJ)$.
\end{lemma}
\begin{proof}
Let $P\in\lambda_X(I)$. Since $I^vJ\subseteq(IJ)^v=(IJ)^v$ \cite[Proposition 32.2(c) and Theorem 34.1(2)]{gilmer}, we have
\begin{equation*}
\nu_{(IJ)^v}(P)\leq\nu_{I^vJ}(P)=\nu_{I^v}(P)+\nu_J(P)<\nu_I(P)+\nu_J(P)=\nu_{IJ}(P).
\end{equation*}
Hence $P\in\lambda_X(IJ)$. If $J$ is invertible, then
\begin{equation*}
\lambda_X(I)\subseteq\lambda_X(IJ)\subseteq\lambda_X(IJ(D:J))=\lambda_X(I)
\end{equation*}
since $J(D:J)=D$. Hence $\lambda_X(I)=\lambda_X(IJ)$.
\end{proof}

\begin{lemma}\label{lemma:psi-residual}
Let $D$ be a strongly discrete Pr\"ufer domain and let $X\subseteq\Max(D)$. Let $I\subseteq D$ be an ideal with $\V(I)\subseteq X$. Then, $I^v=D$ if and only if $\nu_I$ is residual, with respect to the inverse topology.
\end{lemma}
\begin{proof}
We prove that $I^v\neq D$ if and only if $\nu_I$ is not residual.

Suppose that $\nu_I$ is not residual. The set $Z:=\nu_I^{-1}(\insZ\setminus\{0\})$ is equal to the closed set $\V(I)$ of the Zariski topology associated to $I$; hence, $Z$ is closed in the constructible topology, and since the inverse and the constructible topology agree on $\Max(D)$ then $Z$ is closed in $X$, with respect to the inverse topology. Therefore, if $\nu_I$ is not residual there is subset $\Omega\subseteq Z$ that is open in the inverse topology; since a basis of the inverse topology is given by the $\V(J)$, as $J$ ranges among the finitely generated ideals of $D$, we can suppose without loss of generality that $\Omega=\V(J)$ for some finitely generated ideal $J$. By \cite[Theorem 6.10]{InvXD}, $X\setminus\inscrit_X(D)$ is dense in $X$, and thus there is a $P\in\V(J)$ that is not critical; by definition, we can find a radical ideal $J'\in\Inv_X(D)$ contained in $P$. Let $L:=J+J'$: then, $L$ is finitely generated, it is a radical ideal (as it contains $J'$, which is a radical ideal contained only in maximal ideals) and contains $J$, and thus $\V(L)\subseteq\V(J)\subseteq\V(I)$. Hence $I\subseteq L$ and thus $I^v\subseteq L^v=L$. In particular, $I^v\neq D$.

Conversely, if $I^v\neq D$, then $I\subseteq L$ for some finitely generated ideal $L\subsetneq D$; hence, $\V(L)$ is nonempty and open in the inverse topology, and
\begin{equation*}
\V(L)\subseteq\V(I)=\nu_I^{-1}(\insZ\setminus\{0\}).
\end{equation*}
In particular, $\nu_I$ cannot be residual.
\end{proof}

\begin{prop}\label{prop:lambdaX}
Let $D$ be a strongly discrete Pr\"ufer domain and let $X\subseteq\Max(D)$. Let $I$ be a fractional ideal that is evaluable on $X$. Then, $\lambda_X(I)$ is nowhere dense in $X^\inverse$.
\end{prop}
\begin{proof}
If $I^v=D$, then $\nu_I$ is residual by Lemma \ref{lemma:psi-residual}, and thus $\lambda_X(I)=\nu_I^{-1}(\insZ\setminus\{0\})$ is nowhere dense.

If $I$ is arbitrary, then $\lambda_X(I)\subseteq\lambda_X(I(D:I))$, and the latter is nowhere dense by the previous part of the proof since $I$ is $v$-invertible and thus $(I(D:I))^v=D$. Hence also $\lambda_X(I)$ is nowhere dense.
\end{proof}

\begin{prop}\label{prop:Omf}
Let $D$ be a strongly discrete Pr\"ufer domain and let $X\subseteq\Max(D)$ be a splitting set. Then, the map
\begin{equation*}
\begin{aligned}
\Psi\colon \Div_X(D) &\longrightarrow \insfunct(X,\insZ)/\qz(X^\inverse),\\
I & \longmapsto \nu_I+\qz(X^\inverse).
\end{aligned}
\end{equation*}
is an injective group homomorphism.
\end{prop}
\begin{proof}
Let $I,J\in\Div_X(D)$: to show that $\Psi(I\times_vJ)=\Psi(I)+\Psi(J)$, we need to show that $\nu_{I\times_vJ}-(\nu_I+\nu_J)$ is residual. However, $I\times_vJ=(IJ)^v$ and $\nu_I+\nu_J=\nu_{IJ}$; hence,
\begin{equation*}
h:=\nu_{I\times_vJ}-(\nu_I+\nu_J)=\nu_{(IJ)^v}-\nu_{IJ}
\end{equation*}
is a function such that $h^{-1}(\insZ\setminus\{0\})=\lambda_X(IJ)$, which is nowhere dense by Proposition \ref{prop:lambdaX}. Thus $h$ is residual, as claimed, and so $\Psi$ is a group homomorphism.

If $I\in\ker\Psi$, then $\nu_I$ is residual. The proper ideal $I\cap D$ is divisorial, it satisfies $\V(I\cap D)\subseteq X$ and is $v$-invertible \cite[Lemma 4.3]{InvXD}. Since $\nu_{I\cap D}^{-1}(\insZ\setminus\{0\})\subseteq\nu_I^{-1}(\insZ\setminus\{0\})$, the function $\nu_{I\cap D}$ is residual and thus by Lemma \ref{lemma:psi-residual} $I\cap D=(I\cap D)^v=D$. Thus $D\subseteq I$; it follows that $(D:I)\subseteq D$. Since $\Psi$ is a group homomorphism, $\nu_{(D:I)}\in-\nu_I+\qz(X^\inverse)=\qz(X^\inverse)$, i.e., $(D:I)$ is residual; it follows that $(D:I)=(D:I)^v=D$. Thus $I=D$, as claimed.
\end{proof}

\begin{teor}\label{teor:Div}
Let $D$ be a strongly discrete Pr\"ufer domain and let $X\subseteq\Max(D)$.
\begin{enumerate}[(a)]
\item\label{teor:Div:bcrit} If $\bcrit_X(D)=\emptyset$, then $\Div_X(D)$ is free.
\item\label{teor:Div:crit} If $\inscrit_X(D)=\emptyset$, then $\Div_X(D)/\Inv_X(D)$ is free.
\end{enumerate}
\end{teor}
\begin{proof}
\ref{teor:Div:bcrit} Since $\bcrit_X(D)$ is empty, the ideal function $\nu_I$ is bounded for every $I\in\Div_X(D)$ \cite[Lemma 4.4 and Corollary 5.9]{InvXD}, and thus the image of the group homomorphism $\Psi:\Div_X(D)\longrightarrow\insfunct(X,\insZ)/\qz(X^\inverse)$ of Proposition \ref{prop:Omf} is contained in $\displaystyle{\frac{\insbound(X,\insZ)}{\qz(X^\inverse)\cap\insbound(X,\insZ)}=\frac{\insbound(X,\insZ)}{\qz_b(X^\inverse)}}$. Therefore $\Div_X(D)$ is isomorphic to a subgroup of $\insbound(X,\insZ)/\qz_b(X^\inverse)$; as the latter is free by Theorem \ref{teor:quoz-free}, also $\Div_X(D)$ is free.

\ref{teor:Div:crit} Suppose now that $\inscrit_X(D)=\emptyset$. Let $G:=\qz_b(X^\inverse)+\inscont_c(X^\inverse,\insZ)$, and consider the map 
\begin{equation*}
\begin{aligned}
\Psi'\colon \Div(D) &\longrightarrow \insbound(X,\insZ)/G,\\
I & \longmapsto \nu_I+G.
\end{aligned}
\end{equation*}
We claim that $\ker\Psi'=\Inv_X(D)$. Indeed, if $I\in\Inv_X(D)$ then $\nu_I$ is continuous by \cite[Proposition 5.8]{InvXD}, and thus $\nu_I\in G$. Conversely, suppose that $I\in\ker\Psi'$. Then, $\nu_I=h+g$ with $h\in\qz_b(X^\inverse)$ and $g\in\inscont_c(X^\inverse,\insZ)$. By \cite[Lemma 6.4]{InvXD}, there is an invertible ideal $J$ such that $\nu_J=g$; set $J^{-1}:=(D:J)$. Then, $IJ^{-1}$ is a divisorial ideal, which is in $\ker\Psi'$ as $\Inv_X(D)\subseteq\ker\Psi'$. Moreover, $\nu_{IJ^{-1}}\in\qz_b(X^\inverse)$; by the previous part of the proof, $IJ^{-1}=D$, i.e., $I=J$. Hence $I\in\Inv_X(D)$ and $\ker\Psi'=\Inv_X(D)$.

Therefore, $\Psi'$ induces an injective map
\begin{equation*}
\frac{\Div_X(D)}{\Inv_X(D)}\longrightarrow\frac{\insbound(X,\insZ)}{G}.
\end{equation*}
Since the group on the right hand side is free (Theorem \ref{teor:quoz-free}), $\Div_X(D)/\Inv_X(D)$ is free too.
\end{proof}

We single out two special cases.

The first is about almost Dedekind domains, i.e., integral domains that are locally discrete valuation domains. In this case, for $X=\Max(D)$ the ideal function $\nu_I$ is defined for all fractional ideals $I$, and $\nu_I(P)$ coincides with the valuation of $ID_P$. An almost Dedekind domain is \emph{bounded} if every ideal function is bounded. An \emph{SP-domain} is an integral domain such that every ideal can be written as a product of radical ideals; every Dedekind domain is an SP-domain, and every SP-domain is an almost Dedekind domain.
\begin{cor}
Let $D$ be an almost Dedekind domain. Then, the following hold.
\begin{enumerate}[(a)]
\item If $D$ is bounded, $\Div(D)$ is free.
\item If $D$ is an SP-domain, then $\Div(D)/\Inv(D)$ is free.
\end{enumerate}
\end{cor}
\begin{proof}
Let $X:=\Max(D)$. Since $D$ is one-dimensional, $\Div_X(D)=\Div(D)$ and $\Inv_X(D)=\Inv(D)$. If $D$ is bounded, then $\bcrit_X(D)=\emptyset$ and $\Div(D)$ is free by Theorem \ref{teor:Div}. If $D$ is an SP-domain, $\inscrit_X(D)=\inscrit(D)=\emptyset$ \cite[Theorem 2.1]{olberding-factoring-SP} the claim follows again from Theorem \ref{teor:Div}.
\end{proof}
We note that the bounded case of the previous corollary was also proved in \cite{bounded-almded} using completion properties of lattice-ordered groups.

The second case deals with integer-valued polynomials. If $D$ is an integral domain with quotient field $K$, the ring $\Int(D)$ of integer-valued polynomials is the ring of all $f\in K[T]$ such that $f(D)\subseteq D$. A proper ideal $I$ of $\Int(D)$ is \emph{unitary} if $I\cap D\neq(0)$, and a fractional ideal $I$ is \emph{almost-unitary} if $dI\subseteq\Int(D)$ for some $d\in D$. We denote by $\Div^u(\Int(D))$ and $\Inv^u(\Int(D))$  the groups of almost-unitary divisorial and invertible ideals, respectively. We say that $\Int(D)$ \emph{behaves well under localization} if $\Int(D)_P=\Int(D_P)$ for all prime ideals $P$ of $D$.
\begin{cor}
Let $D$ be an SP-domain such that every residue field is finite and $\Int(D)$ behaves well under localization. Then, the groups $\Div^u(\Int(D))$ and $\Div^u(\Int(D))/\Inv^u(\Int(D))$ are free.
\end{cor}
\begin{proof}
Since $D$ has finite residue fields and $\Int(D)$ behaves well under localization, it is a Pr\"ufer domain \cite[Proposition VI.1.6]{intD}. Let $X$ be the set of all unitary maximal ideals. Then, $\inscrit_X(D)=\emptyset$: indeed, if $M\in X$, then $M\cap D$ contains a finitely generated radical ideal $I$, and thus $I\cdot\Int(D)$ is a finitely generated radical ideal contained in $M$. Thus $\nu_I(P)\leq 1$ for all $P\in X$, and so $M$ is not critical. Since $\Div^u(D)=\Div_X(D)$ and $\Inv^u(D)=\Inv_X(D)$, the claim follows from Theorem \ref{teor:Div}.
\end{proof}

\bibliographystyle{plain}
\bibliography{/bib/articoli,/bib/libri,/bib/miei}
\end{document}